\documentclass{amsproc}
\usepackage{amsmath}
\usepackage[a4paper,margin=3.5cm]{geometry}
\usepackage{amssymb,amsmath,amssymb,color,enumerate,amssymb,wasysym,mathrsfs,comment}
\usepackage{tikz}
\usetikzlibrary{positioning}
\input xy
\xyoption{all}

\newtheorem{theorem}{Theorem}[section]

\theoremstyle{definition}

\newtheorem{example}[theorem]{Example}
\theoremstyle{plain}

\newtheorem{corollary}[theorem]{Corollary}
\newtheorem{Prop}[theorem]{Proposition}
\newtheorem{Thm}[theorem]{Theorem}
\newtheorem{Setting}[theorem]{Setting}

\theoremstyle{remark}
\newtheorem{remark}[theorem]{Remark}

\numberwithin{equation}{section}
\newcommand{\Cal}[1]{{\mathcal #1}}

\newcommand{\Spec}{\operatorname{Spec}}

\newcommand{\sub}[1]{\mathrm{Sub} (#1)}

\newcommand{\cmat}{\left(\begin{array}}
\newcommand{\fmat}{\end{array}\right)}

\newcommand{\colim}{\varinjlim}
\newcommand{\Fin}{\operatorname{Fin}}
\newcommand{\Alg}{\operatorname{Alg}}
\newcommand{\Sh}{\operatorname{Sh}}
\newcommand{\Clop}{\operatorname{Clop}}

\usepackage[mathscr]{euscript}

\title{Spaces of subobjects as spectral spaces}
\author[F. Campanini]{Federico Campanini}
\address{Universit\'e catholique de Louvain, Institut de Recherche en Math\'ematique et Physique, 1348 Louvain-la-Neuve, Belgium}
\email{federico.campanini@uclouvain.be}

\author[C. A. Finocchiaro]{Carmelo Antonio Finocchiaro}
\address{Universit\`a degli Studi di Catania, Dipartimento di Matematica e Informatica,
Viale Andrea Doria 6, 95125 Catania, Italy}
\email{cafinocchiaro@unict.it}

\thanks{The first named author was a postdoctoral researcher of the Fonds de la Recherche Scientifique - FNRS when the project started. The second named author was supported by UCLouvain Bureau SST -- ``Professeurs et chercheurs visiteurs'' 2025, by GNSAGA, by the research project PIACERI ``ACIVA - Anelli commutativi, loro
ideali e varietà algebriche'' and by the research project PRIN 2022 ``Unirationality,
Hilbert schemes, and singularities''.}

\begin{document}

\begin{abstract}
We study natural topologies on spaces of subobjects of a fixed object in a suitable category, with the aim of determining when these spaces are spectral. A central step in our approach, which is also of independent interest, is the comparison between the categorical notion of a $\lambda$-generated object and the order-theoretic notion of a $\lambda$-compact element in a lattice of subobjects. We prove that these notions coincide under some natural and mild assumptions . In the finitary case, this allows us to describe the finitely generated subobjects purely in order-theoretic terms and to construct, inside each interval of a subobject lattice, a canonical algebraic core. We study this construction abstractly for complete lattices and characterize it by a universal property. We then introduce the categorical Zariski topology on spaces of subobjects, relate it to the Scott topology, and obtain spectrality criteria for the whole subobject space and for its algebraic core.
\end{abstract}

\maketitle

\section{Introduction}

Following Hochster \cite{hochster_spectral}, a topological space $X$ is said to be {\em spectral} if it is homeomorphic to the prime spectrum of a commutative unital ring endowed with the Zariski topology. In the main result of Hochster's thesis, spectral spaces were characterized in purely topological terms. More precisely, a topological space $X$ is spectral if and only if it satisfies the following conditions:
\begin{enumerate}
    \item
    $X$ is $T_0$ and quasi-compact;
    \item
    the collection of quasi-compact open subspaces of $X$ is a basis which is closed under finite intersections;
    \item
    $X$ is sober, meaning that any irreducible closed subset of $X$ is the closure of a unique point.
\end{enumerate}

In the last few decades, spectral spaces have turned out to be a powerful tool for addressing questions coming from various research areas in algebra, such as commutative ring theory and multiplicative ideal theory. For instance, a key tool for investigating ideal-theoretic features of certain classes of integral domains is the abstract Riemann surface of valuation domains of a given field $k$, which, when endowed with a topology introduced by Zariski, is a prominent example of a spectral space. For a deeper insight into this circle of ideas, see \cite{fifolo_transactions}.

Many examples of spectral spaces can be found in the context of commutative algebra. For instance, one may consider the space of subrings of a ring, the space of ideals (or of radical ideals) of a ring, and related spaces (see \cite{finocchiaro-ultrafiltri,top-vers-null,fi-fo-sp-dist,olberding_irredundant}).

Motivated by a recent paper by A. Banerjee \cite{Banerjee}, we are interested in understanding when certain families of subobjects of a given object in a suitable category are spectral spaces. We show that the setting proposed in \cite{Banerjee} can be significantly generalized so as to include, for instance, any variety of algebras or any Grothendieck topos. It is worth noting that the techniques we use are different from those in \cite{Banerjee}.

More precisely, consider a well-powered locally small category $\mathscr A$ admitting filtered colimits. Then, given an object $X \in \mathscr{A}$, the set $\sub X$ of all subobjects of $X$ can be endowed with a natural topological structure, which we call the {\em categorical Zariski topology}, generated by the sets of the form
$$
\mathcal B_F:=\{I\in \sub X\mid F\subseteq I\},
$$
where $F$ ranges over the finitely generated subobjects of $X$.

A crucial aspect of our investigation concerns the fact that, in categories satisfying a very natural and mild setting (see Setting~\ref{setting}), the categorical notion of finitely generated object can be seen in purely order-theoretical terms, namely in terms of compact elements; see Proposition~\ref{compact=fg}. More generally, for every regular cardinal $\lambda$, we compare $\lambda$-generated objects with $\lambda$-compact elements in lattices of subobjects. It is worth noting that our result links a global notion (that of $\lambda$-generated object) with a local one (that of $\lambda$-compact element in a poset of subobjects).

Starting from this comparison, we isolate an order-theoretic construction associated with a complete lattice $X$ and an element $a\in X$, which we call the algebraic core of $X$ over $a$. We prove that this core is an algebraic lattice and characterize it as a coreflective subposet of the interval $[a,1]$. In the categorical setting, this construction recovers and extends the family $\Fin(B,A)$ considered by Banerjee in the context of AB5 abelian categories.

We then study the categorical Zariski topology on spaces of subobjects. Under Setting~\ref{setting} for $\lambda=\aleph_0$, we show that the topology is $T_0$ precisely when the corresponding subobject lattice is algebraic; in this case it coincides with the Scott topology and is spectral. Even when the whole lattice of subobjects is not algebraic, the categorical algebraic core is always an algebraic lattice and its categorical Zariski topology is spectral. We conclude with examples and applications, including varieties of algebras, categories of sheaves, and an interpretation of a proper algebraic core in terms of the space of connected components of a compact space.

\section{Basics and terminology}
We recall some basic notions concerning categories and partially ordered sets. We shall mainly follow the terminology of \cite{adamek-rosicky} and \cite{Gierz2003}.

\subsection{Preliminaries on categories}

Let $\mathscr{A}$ be a category. 
\begin{enumerate}
    \item 
    $\mathscr{A}$ is {\em locally small} if $\hom(A,B)$ is a set for every pair of objects $A,B \in \mathscr{A}$.
    \item 
    $\mathscr{A}$ is {\em well-powered} if $\sub X$ is a set for every $X \in \mathscr{A}$.
    \item 
    $\mathscr{A}$ is complete (resp. cocomplete) if it admits all (small) limits (resp. colimits).
    \item
    Let $\lambda$ be a regular cardinal. A small category $I$ is said to be \emph{$\lambda$-filtered} if it satisfies the following conditions:
        \begin{enumerate}
        \item $I$ is nonempty;
        \item for every set of objects $\{i_j\}_{j \in J}$ of $I$ with $|J| < \lambda$, there exists an object $i \in I$ together with morphisms $i_j \to i$ for all $j \in J$;
        \item for every set of morphisms $\{u_j : i \to i'\}_{j \in J}$ in $I$ with $|J| < \lambda$, there exists a morphism $v : i' \to k$ such that $v \circ u_j = v \circ u_{j'}$ for all $j,j' \in J$.
    \end{enumerate}
    A \emph{$\lambda$-filtered colimit} in $\mathcal{A}$ is a colimit indexed by a $\lambda$-filtered category.
    If the category $I$ is a preordered set $(I,\leq)$, we shall often use the term $\lambda$-directed colimit.
    In particular, for $\lambda = \aleph_0$, one recovers the usual notion of filtered (or directed) category and filtered (or directed) colimit.
    \item 
    Assume that $\mathscr{A}$ is locally small and admits $\lambda$-filtered colimits of monomorphisms. An object $X \in \mathscr{A}$ is called {\em  $\lambda$-generated} provided that the hom-functor
        $$
        \hom(X, -) \colon \mathscr{A} \to \mathsf{Set}
        $$
    preserves $\lambda$-filtered colimits of monomorphisms \cite[Definition 1.67]{adamek-rosicky}. Notice that $\lambda$-filtered colimits can be replaced with $\lambda$-directed colimits, see~\cite[Theorem~1.5 and Remark~1.21]{adamek-rosicky}. In particular, $\lambda$-generated objects admit the following explicit description. An object $X \in \mathscr{A}$ is $\lambda$-generated if for every $\lambda$-directed diagram $\{u_{ij}\colon A_i\to A_j \mid i,j \in I, i\leq j\}$ of monomorphisms, with colimit cocone $\{\mu_i\colon A_i\to \colim A_i\}_{i\in I}$, the following conditions hold:
    \begin{enumerate}
    \item for every morphism $f\colon X\to \colim A_i$, there exist $i\in I$ and a morphism $f_i\colon X\to A_i$ such that $f=\mu_i f_i$;
    \item if $i,j\in I$ and $f_i\colon X\to A_i$ and $f_j\colon X\to A_j$ satisfy $\mu_i f_i=\mu_j f_j$, then there exists $k\in I$, with $i,j\leq k$, such that $u_{ik}f_i=u_{jk}f_j$.
    \end{enumerate}
    If $\lambda=\aleph_0$, we shall talk about {\em finitely generated objects}.
    \item 
    Let $\lambda$ be a regular cardinal. An object $X$ of a locally small category $\mathcal{A}$ is called {\em $\lambda$-presentable} if the hom-functor
    $$
    \mathrm{Hom}_{\mathcal{A}}(X,-) : \mathcal{A} \to \mathbf{Set}
    $$
    preserves $\lambda$-filtered colimits \cite[Definition~1.13]{adamek-rosicky}.
    A category $\mathcal{A}$ is said to be {\em locally $\lambda$-presentable} if it is cocomplete and there exists a set $\mathcal{P}$ of $\lambda$-presentable objects such that every object of $\mathcal{A}$ can be expressed as a $\lambda$-filtered colimit of objects from $\mathcal{P}$. Finally, $\mathcal{C}$ is called {\em locally presentable} if it is locally $\lambda$-presentable for some regular cardinal $\lambda$ \cite[Chapter~1]{adamek-rosicky}. Equivalently, locally presentable categories are precisely those categories that are accessible and cocomplete \cite[Corollary~2.47]{adamek-rosicky}.
    \item 
    If $\mathscr{A}$ is an abelian category, $\mathscr{A}$ satisfies the axiom AB5 if all colimits exist and filtered colimits are exact \cite[Chapter~2.8]{Popescu-abelian}.
    \end{enumerate}

\subsection{Preliminaries on partial orders}\label{order-definitions}
Let $(X,\leq)$ be a partially ordered set and let $\lambda$ be a regular cardinal.
\begin{enumerate}
    \item
    For every $a,b \in X$, we denote the interval between $a$ and $b$ by
    $$
    [a,b]:=\{x \in X \mid a\leq x\leq b\}.
    $$
    \item
    An element $c\in X$ is \emph{compact} if, for every directed subset $\Delta\subseteq X$ such that the supremum $\bigvee\Delta$ exists, the inequality $c\leq\bigvee\Delta$ implies that $c\leq d$ for some $d\in\Delta$. We denote by $K_X$ the set of compact elements of $X$.
    \item
    An element $c\in X$ is \emph{$\lambda$-compact} if, for every $\lambda$-directed subset $\Delta\subseteq X$ such that the supremum $\bigvee\Delta$ exists, the inequality $c\leq\bigvee\Delta$ implies that $c\leq d$ for some $d\in\Delta$. Here, $\Delta$ is $\lambda$-directed if every subset of $\Delta$ of cardinality smaller than $\lambda$ has an upper bound in $\Delta$.
    \item
    A complete lattice $X$ is \emph{algebraic} if every element $x\in X$ is the supremum of the compact elements below it, that is,
    $$
    x=\bigvee\{c\in K_X\mid c\leq x\}.
    $$
    \item 
    We say that $X$ is a {\em directed-complete partial order} (dcpo, for short), if every non-empty directed subset of $X$ has a supremum in $X$. Recall that the \emph{Scott topology} on a dcpo $X$ consists of the upper subsets $U$ such that, for every directed subset $\Delta\subseteq X$, the relation $\bigvee\Delta\in U$ implies $\Delta\cap U\neq\varnothing$.
    \item
    If $X$ is an algebraic lattice, the family of subsets 
    $$
    \Cal B_c:=\{x \in X \mid c\leq x\},
    $$
    where $c$ runs over the compact elements of $X$, is a basis for the Scott topology on $X$. We shall use the classical fact that the Scott space of an algebraic lattice is spectral; see, for instance, \cite[Chapter~II]{Gierz2003}.
\end{enumerate}

\section{$\lambda$-generated objects and $\lambda$-compact elements}

Let $\Cal A$ be a well-powered category admitting filtered colimits. We aim to compare the notions of finitely generated object and compact element in the posets of subobjects in a very natural and mild setting. More generally, we are actually in a position to compare the notions of $\lambda$-generated object and $\lambda$-compact element, where $\lambda$ is a regular cardinal.

It is worth noting that the comparison we are going to present links a global notion (that of $\lambda$-generated object) with a local one (that of $\lambda$-compact element).
The setting we want to consider is the following.

\begin{Setting}\label{setting} Fix a regular cardinal $\lambda$. Let $\mathscr{A}$ be a complete and well-powered category admitting $\lambda$-filtered colimits. Assume that $\mathcal{A}$ satisfies the following properties.
\begin{enumerate}
    \item
    $\lambda$-directed unions coincide with $\lambda$-filtered colimits, in the following sense: 
    if an object $X \in \mathcal{C}$ is the union of a $\lambda$-direct family of subobjects $\{u_i : A_i \hookrightarrow X\}_{i \in I}$ of $X$, then the diagram of all $A_i$'s and the factorizations of $u_i$ through $u_j$ for $i\leq j$ has a colimit canonically isomorphic to $X$.
    \item
    $\lambda$-filtered colimits commute with pullbacks.
    \item 
    For every $\lambda$-filtered colimit of monomorphisms, the colimit cocone consists of monomorphisms, that is, for every $\lambda$-filtered diagram of monomorphisms $\{u_{i,j}\colon A_i\to A_j \mid i,j \in I\}$, the structural morphisms
    \[
    A_i \longrightarrow \varinjlim_{i \in I} A_i
    \]
    are monomorphisms for all $i \in I$.
\end{enumerate}
\end{Setting}

\begin{remark}
\begin{enumerate}
    \item
    If $\Cal A$ is a locally $\lambda$-presentable category for some regular cardinal $\lambda$, then $\Cal A$ satisfies Setting~\ref{setting} for that $\lambda$ (see \cite[Chapter~1]{adamek-rosicky}). In particular, every variety of algebras and every Grothendieck topos satisfy Setting~\ref{setting} for a suitable regular cardinal $\lambda$.
    \item 
    Another class of examples is given by complete AB5 abelian categories (see \cite[Chapter~2, Theorem~8.6]{Popescu-abelian}).
    \item
    Recall that if an accessible category $\Cal A$ is complete, then it is locally presentable \cite[Corollary~2.47]{adamek-rosicky}. Therefore, since completeness is part of Setting~\ref{setting}, examples outside the locally presentable setting are necessarily non-accessible and typically require a change of universe or ``genuinely large" algebraic data.
    \item
    Every locally small cocomplete elementary topos satisfies Setting~\ref{setting} for every regular cardinal $\lambda$. This follows from the standard exactness and universality properties of colimits in elementary topoi~\cite[Chapter~5]{BorceuxHandbook3}. Consequently, cocomplete elementary topoi which are not Grothendieck topoi provide examples satisfying the setting which are not locally presentable.
\end{enumerate}
\end{remark}

\begin{Prop}\label{compact=fg} Fix a regular cardinal $\lambda$. Let $\mathscr A$ be a category satisfying Setting~\ref{setting} for that $\lambda$, and let $X$ be an object of $\mathscr A$. Then the following conditions are equivalent. 
    \begin{enumerate}
        \item $X$ is a $\lambda$-generated object of $\mathscr A$.
        \item For every object $Y$ of $\mathscr{A}$ such that $X$ is a subobject of $Y$, $X$ is $\lambda$-compact as an element of the partially ordered set $\sub Y$.
        \item 
        $X$ is $\lambda$-compact as an element of the partially ordered set $\sub X$.
        \item 
        There exists an object $Y$ of $\mathscr A$ such that $X$ is $\lambda$-compact as an element of the partially ordered set $\sub Y$.
    \end{enumerate}
\end{Prop}

\begin{proof}
    $(1)\Rightarrow(2)$. Let $Y$ be an object of $\mathscr A$ such that $X\subseteq Y$. Consider a $\lambda$-directed system $\mathcal G$ of subobjects of $Y$ such that $X\subseteq \bigvee \mathcal G$. By hypothesis, $\colim \mathcal G$ is canonically isomorphic to $\bigvee \mathcal G$ (Property $(1)$). Let $\mu_\Gamma:\Gamma\to \colim \mathcal G$ denote the structural morphism, for all $\Gamma\in\mathcal G$. Since $X$ is $\lambda$-generated, any embedding $\iota:X\to \bigvee \mathcal G$ factors through  $\mu_\Gamma$, for some $\Gamma\in\mathcal G$, say $\iota=\mu_\Gamma h$, for a unique morphism $h:X\to \Gamma$. Since $\iota$ is a monomorphism, it follows that $h$ is a monomorphism, that is,  $X$ is a subobject of $\Gamma$. This shows that $X$ is $\lambda$-compact in $\sub Y$. 

    The implications $(2)\Rightarrow(3) \Rightarrow (4)$ are clear.
    
    $(4)\Rightarrow(1)$. Let $\Cal D:= \{u_{i,j}\colon A_i\to A_j\mid i,j \in I\}$ be a $\lambda$-directed system of monomorphisms and let $Z:= \colim A_i$ be the colimit of the system. By hypothesis, all the structural morphisms $\mu_i \colon A_i \to Z$ of the colimit cocone are monomorphisms; therefore, $\Cal D$ defines a directed subset of the poset $\sub Z$, whose union is equal to $Z$, that is, $\bigvee_{i \in I} A_i=Z$.
    Let $\alpha \colon X \to Z$ be a morphism in $\mathscr{A}$ and consider the family of monomorphisms $\{\mu^*_i \colon \alpha^{-1}(A_i)\to X\mid i \in I\}$ obtained by pulling back the morphisms $\mu_i \colon A_i \to Z$, $i \in I$, along $\alpha$. Moreover, the family $\{u_{i,j}\colon A_i\to A_j\mid i,j \in I\}$ induces a directed system
    $$
    \Cal D^*:=\{u^*_{i,j}\colon \alpha^{-1}(A_i)\to \alpha^{-1}(A_j)\mid i,j \in I\}
    $$
    of monomorphisms between subobjects of $X$. Since $\lambda$-directed unions coincide with $\lambda$-directed colimits and $\lambda$-directed colimits commute with pullbacks, $\bigvee \alpha^{-1}(A_i)=\alpha^{-1}\big( \bigvee A_i \big)=X$. Since $X$ is a subobject of $Y$, all the objects $\alpha^{-1}(A_i)$ for $i \in I$ are also subobjects of $Y$ and, since $X$ is $\lambda$-compact in $\sub Y$, there exists an index $k \in I$ such that $X=\alpha^{-1}(A_k)$ (to be more precise, there exists an index $k \in I$ such that $\mu_k^\ast\colon \alpha^{-1}(A_k)\to X$ is an isomorphism). This gives a unique decomposition of $\alpha$ through $\mu_k$, as desired.
\end{proof}

\begin{remark}
    For $\lambda=\aleph_0$, the equivalence $(1) \Leftrightarrow (2)$ is already known in the context of Grothendieck categories (see, for instance, \cite[Chapter V, Section 3]{Sten}).
\end{remark}

\section{Algebraic cores and the categorical Zariski topology}

We first isolate an order-theoretic construction that will later be applied to subobject lattices.

\subsection{The algebraic core of a complete lattice}

Let $X$ be a complete lattice and let $a\in X$. We define the {\em algebraic core of $X$ over $a$} as:
$$
\Alg_a(X):=
\left\{
 a\vee\bigvee\Sigma
 \ \middle|\ 
 \Sigma\subseteq K_X
\right\}.
$$
The following result justifies the terminology.

\begin{Thm}\label{abstract-algebraic-core}
Let $X$ be a complete lattice and let $a\in X$. Then the following statements hold.
\begin{enumerate}
    \item $\Alg_a(X)$ is the smallest subset of the interval $[a,1]$ that contains every element $a\vee c$, with $c\in K_X$, and is closed under arbitrary non-empty suprema computed in $X$.
    \item The inclusion
    $$
    i:\Alg_a(X)\longrightarrow[a,1]
    $$
    admits a right adjoint
    $$
    \rho_a:[a,1]\longrightarrow\Alg_a(X),
    \qquad
    \rho_a(x):=a\vee\bigvee\{c\in K_X\mid c\leq x\}.
    $$
    Equivalently, for $y\in\Alg_a(X)$ and $x\in[a,1]$,
    $$
    i(y)\leq x
    \quad\Longleftrightarrow\quad
    y\leq\rho_a(x).
    $$
    Thus $\rho_a(x)$ is the largest element of $\Alg_a(X)$ below $x$.
    \item $\Alg_a(X)$ is an algebraic lattice. Arbitrary non-empty suprema are computed as in $X$, while
    $$
    \bigwedge_{\Alg_a(X)}x_j
    =
    \rho_a\left(\bigwedge_Xx_j\right).
    $$
    The compact elements of $\Alg_a(X)$ are precisely the elements of the form $a\vee c$, with $c\in K_X$.
\end{enumerate}
\end{Thm}

\begin{proof}
$(1)$ If $\{x_j\}_{j\in J}$ is a non-empty family in $\Alg_a(X)$ and $x_j=a\vee\bigvee\Sigma_j$, then
$$
\bigvee_{j\in J}x_j
=
 a\vee\bigvee\left(\bigcup_{j\in J}\Sigma_j\right),
$$
so $\Alg_a(X)$ is closed under arbitrary non-empty suprema in $X$. Of course, for the empty family we recover $a$ as the bottom element of $\Alg_a(X)$. Moreover,
$$
a\vee\bigvee\Sigma
=
\bigvee_{c\in\Sigma}(a\vee c).
$$
This proves the minimality statement in~(1).

$(2)$ Let $y\in\Alg_a(X)$ and write $y=a\vee\bigvee\Sigma$, with $\Sigma\subseteq K_X$. If $y\leq x$, then every $c\in\Sigma$ satisfies $c\leq x$, and hence
$$
y=a\vee\bigvee\Sigma
\leq
 a\vee\bigvee\{c\in K_X\mid c\leq x\}
=
\rho_a(x).
$$
Conversely, $\rho_a(x)\leq x$, since $a\leq x$ and every compact element occurring in its definition is below $x$. Therefore
$$
i(y)\leq x
\quad\Longleftrightarrow\quad
y\leq\rho_a(x),
$$
which proves~(2).

$(3)$ Since $\Alg_a(X)$ is closed under arbitrary suprema and has minimum $a$, it is a complete lattice. The formula for its infima follows from the adjunction. Let $c\in K_X$. We claim that $a\vee c$ is compact in $\Alg_a(X)$. Let $\Delta\subseteq\Alg_a(X)$ be directed and assume that
$$
a\vee c\leq\bigvee\Delta.
$$
Since suprema in $\Alg_a(X)$ are computed in $X$, we have $c\leq\bigvee\Delta$. The compactness of $c$ in $X$ yields $d\in\Delta$ such that $c\leq d$. Since $a\leq d$, it follows that $a\vee c\leq d$.

Every element $x=a\vee\bigvee\Sigma$ of $\Alg_a(X)$ satisfies
$$
x=\bigvee_{c\in\Sigma}(a\vee c),
$$
so $\Alg_a(X)$ is algebraic. Finally, if $x$ is compact in $\Alg_a(X)$, consider the directed family
$$
\left\{a\vee\bigvee F\ \middle|\ F\subseteq\Sigma\text{ finite}\right\},
$$
whose supremum is $x$. Compactness gives a finite subset $F\subseteq\Sigma$ such that
$$
x=a\vee\bigvee F.
$$
Since a finite join of compact elements of $X$ is compact, $x=a\vee c$ for some $c\in K_X$. This proves~(3).
\end{proof}

\begin{remark}
Notice that $\Alg_a(X)$ need not be a sublattice of $X$, since infima in $\Alg_a(X)$ need not coincide with those computed in $X$. Moreover, $X$ is an algebraic lattice if and only if $X=\Alg_0(X)$.
\end{remark}

\subsection{The categorical Zariski topology}
Let $\mathscr A$ be a locally small and well-powered category admitting filtered colimits of monomorphisms, and let $A\in\mathscr A$. We define the \emph{categorical Zariski topology} on $\sub A$ as the topology generated by the sets
$$
\Cal B_F:=\{T\in\sub A\mid F\subseteq T\},
$$
where $F$ ranges over the finitely generated subobjects of $A$.

The terminology comes from the classical prime spectrum. Let $R$ be a commutative ring and let $F=(f_1,\ldots,f_n)$ be a finitely generated ideal. On $\Spec(R)$, the condition $F\subseteq\mathfrak p$
defines the Zariski closed subset $V(F)=V(f_1,\ldots,f_n)$.
Equivalently, $V(F)$ is a quasi-compact open subset of the inverse topology of $\Spec(R)$. 
A further motivation for our construction comes from another Zariski-type topology that plays an important role in ring theory. Given a ring extension $R\subseteq S$, let $\Gamma(S/R)$ denote the space of intermediate rings. Its Zariski topology admits as a basis the sets of the type
$\{T\in\Gamma(S/R)\mid R[F]\subseteq T\}$, where $F$ ranges over the finite subsets of $S$~\cite[Proposition~3.5]{finocchiaro-ultrafiltri}.
These examples motivate our definition: the categorical topology is likewise generated by finite containment conditions. It also satisfies a number of elementary properties analogous to those arising in these classical settings, which we collect below.

\begin{Prop}\label{basic-properties-zariski}
Let $\mathscr A$ be a locally small and well-powered category admitting filtered colimits of monomorphisms, and let $A\in\mathscr A$. Endow $\sub A$ with the categorical Zariski topology. Then the following statements hold.
\begin{enumerate}
    \item Every open subset of $\sub A$ is an upper subset.
    \item For every finitely generated subobject $F\subseteq A$, the open subset $\Cal B_F$ is quasi-compact.
    \item For all $T,U\in\sub A$, one has
    $$
    T\in\overline{\{U\}}
    \quad\Longleftrightarrow\quad
    \text{every finitely generated subobject of $T$ is contained in $U$}.
    $$
    \item If finite joins of finitely generated subobjects of $A$ exist and are finitely generated, and the least subobject of $A$ is finitely generated, then the collection
    $$
    \{\Cal B_F\mid F\subseteq A\text{ is finitely generated}\}
    $$
    is a basis of quasi-compact open subsets closed under finite intersections.
\end{enumerate}
\end{Prop}

\begin{proof}
$(1)$ For every finitely generated subobject $F\subseteq A$, the set $\Cal B_F$ is an upper subset of $\sub A$. Since arbitrary unions and finite intersections of upper subsets are upper subsets, the conclusion follows.

$(2)$ Let $F\subseteq A$ be finitely generated and suppose that
$$
\Cal B_F\subseteq\bigcup_{i\in I}U_i
$$
is an open covering. Since $F\in\Cal B_F$, there exists $i\in I$ such that $F\in U_i$. By~(1), $U_i$ is an upper subset. Hence every subobject of $A$ containing $F$ belongs to $U_i$, and therefore $\Cal B_F\subseteq U_i$. Thus $\Cal B_F$ is quasi-compact.

$(3)$ Let $T,U\in\sub A$. Assume first that $T\in\overline{\{U\}}$, and let $F\subseteq T$ be finitely generated. Then $T\in\Cal B_F$, so every open neighbourhood of $T$, and in particular $\Cal B_F$, meets $\{U\}$. Hence $U\in\Cal B_F$, that is, $F\subseteq U$.

Conversely, assume that every finitely generated subobject of $T$ is contained in $U$. Let $V$ be an open neighbourhood of $T$. Since the sets $\Cal B_F$ form a subbasis, there exist finitely generated subobjects $F_1,\ldots,F_n\subseteq A$ such that
$$
T\in\Cal B_{F_1}\cap\cdots\cap\Cal B_{F_n}\subseteq V.
$$
For every $j$, the inclusion $T\in\Cal B_{F_j}$ means that $F_j\subseteq T$, and hence $F_j\subseteq U$ by hypothesis. Therefore
$$
U\in\Cal B_{F_1}\cap\cdots\cap\Cal B_{F_n}\subseteq V.
$$
Thus every open neighbourhood of $T$ contains $U$, and so $T\in\overline{\{U\}}$.

$(4)$ Assume that finite joins of finitely generated subobjects exist and are finitely generated. Then, for finitely generated subobjects $F,G\subseteq A$,
$$
\Cal B_F\cap\Cal B_G=\Cal B_{F\vee G}.
$$
Moreover, if $0_A$ denotes the least subobject of $A$, then $0_A$ is finitely generated by assumption and $\Cal B_{0_A}=\sub A$. Hence the sets $\Cal B_F$ form a basis closed under finite intersections. Moreover, by~(2), every member of this basis is quasi-compact.
\end{proof}

\begin{remark}
    Notice that the categorical Zariski topology may fail to be $T_0$, and hence
need not be spectral. Consider for example the real unit interval $I=[0,1]$ with the Euclidean topology. When seen as a category, $I$ is locally small and well-powered and admits filtered colimits of monomorphisms. Clearly, we have $\sub 1\cong I$. Moreover, since $I$ is a poset, the finitely generated objects are precisely the compact elements of $I$ (because the morphisms coincide with the relations $\leq$ ). Since the only compact element of $I$ is $0$, the categorical Zariski topology on $I$ is the indiscrete topology.
\end{remark}

\begin{example}\label{example-t0-specialization}
Let $(P,\leq)$ be the poset whose underlying set is
$$
P=\{0,a,b,x,d_0,d_1,\ldots,1\},
$$
where $0$ and $1$ are the bottom and top element, respectively, and whose order is generated by
$$
a,b<x \qquad \text{and} \qquad  
a,b<d_0<d_1<\cdots,
$$
and $x$ is incomparable with $d_n$ for every $n \in \mathbb N$. Its Hasse diagram is the following.

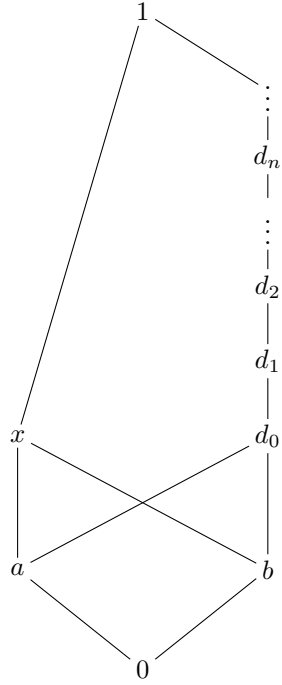
\begin{figure}[ht]
\centering
\begin{tikzpicture}[scale=1.1, every node/.style={inner sep=2pt}]
    \node (zero) at (0,0) {$0$};

    \node (a) at (-1.5,1.2) {$a$};
    \node (b) at ( 1.5,1.2) {$b$};

    \node (x)  at (-1.5,2.8) {$x$};

    \node (d0) at (1.5,2.8) {$d_0$};
    \node (d1) at (1.5,3.7) {$d_1$};
    \node (d2) at (1.5,4.6) {$d_2$};
    \node (vd1) at (1.5,5.35) {$\vdots$};
    \node (dn) at (1.5,6.15) {$d_n$};
    \node (vd2) at (1.5,6.95) {$\vdots$};

    \node (one) at (0,7.9) {$1$};

    \draw (zero) -- (a);
    \draw (zero) -- (b);

    \draw (a) -- (x);
    \draw (b) -- (x);

    \draw (a) -- (d0);
    \draw (b) -- (d0);

    \draw (d0) -- (d1);
    \draw (d1) -- (d2);
    \draw (d2) -- (vd1);
    \draw (vd1) -- (dn);
    \draw (dn) -- (vd2);

    \draw (x) -- (one);
    \draw (vd2) -- (one);
\end{tikzpicture}
\end{figure}

Regarded as a category, $P$ is locally small and well-powered and admits filtered colimits of monomorphisms. Taking $A=1$, we have $\sub A\cong P$. Moreover, since $P$ is a poset, the finitely generated objects are precisely the compact elements of $P$. In particular, from the fact that $x\leq 1=\bigvee_{i \in \mathbb N}d_i$ and $x\nleq d_n$ for every $n \in \mathbb N$, it follows that the element $x$ is not finitely generated. Clearly, the same argument applies to the top element 1. On the other hand, it is easy to check that every other element of $P$ is finitely generated. Therefore the set of finitely generated subobjects of $1$  (i.e., the set of compact elements of $P$) is
$$
K_P=\{0,\ a,\ b,\ d_0,d_1,\ldots\}.
$$
We want to prove that the categorical Zariski topology on $\sub A$ is spectral. First of all, a direct verification based on the description of $K_P$ shows that the space is $T_0$. Now, set
$$
U_a=\Cal B_a,\qquad
U_b=\Cal B_b,\qquad
U_x=U_a\cap U_b,\qquad
U_n=\Cal B_{d_n}.
$$
Then
$$
U_x=\{x,d_0,d_1,\ldots,1\},
\qquad
U_n=\{d_m\mid m\geq n\}\cup\{1\}.
$$
The open subsets of $P$ are precisely
$$
\varnothing,\quad P,\quad
U_x,\quad U_a,\quad U_b,\quad U_a\cup U_b,\quad
\text{and} \quad
U_n \text{ for every } n \in \mathbb N .
$$
Hence $\{P,U_a,U_b,U_x,U_0, U_1\dots \}$ is a basis closed under finite intersections. Indeed,
$$
U_a\cap U_b=U_x,\qquad
U_x\cap U_n=U_n,\qquad
U_n\cap U_m=U_{\max\{n,m\}},
$$
and the remaining intersections are immediate. Moreover, each member of this basis is quasi-compact by Proposition \ref{basic-properties-zariski} (2) and a direct verification on $U_x$.

Finally, the non-empty closed subsets are
$$
\{0\},\qquad
\{0,a\},\qquad
\{0,b\},\qquad
\{0,a,b\},\qquad
\{0,a,b,x\},
$$
$$
\{0,a,b,x,d_0,\ldots,d_n\}\quad \text{ for every } n\in\mathbb N,
\qquad
P.
$$
Among these,
$$
\{0,a,b\}=\{0,a\}\cup\{0,b\}
$$
is reducible, while every other nonempty closed subset is the closure of a point. More precisely,
$$
\overline{\{0\}}=\{0\},\qquad
\overline{\{a\}}=\{0,a\},\qquad
\overline{\{b\}}=\{0,b\}, \qquad
\overline{\{x\}}=\{0,a,b,x\},
$$
$$
\overline{\{d_n\}}=\{0,a,b,x,d_0,\ldots,d_n\}, \qquad \text{and}\qquad 
\overline{\{1\}}=P.
$$
Thus every nonempty irreducible closed subset has a generic point, which is unique since the space is $T_0$. Therefore the space is sober. Since $P$ is quasi-compact and admits a basis of quasi-compact open subsets closed under finite intersections, it is spectral.

Notice that, in this example, finitely generated objects correspond to compact elements even though $P$, when regarded as a category, does not satisfy Setting~\ref{setting} for $\lambda=\aleph_0$ (indeed, $P$ is not a complete lattice).
\end{example}

From now on, we assume that $\mathscr A$ satisfies Setting~\ref{setting} for $\lambda=\aleph_0$. By Proposition~\ref{compact=fg}, specialized to $\lambda=\aleph_0$, the finitely generated subobjects of $A$ are precisely the compact elements of $\sub A$. Since finite joins of compact elements are compact, the sets $\Cal B_F$, with $F$ finitely generated, form a basis of the categorical Zariski topology. Moreover, if $\sub A$ is algebraic, then the categorical Zariski topology on $\sub A$ coincides with its Scott topology. More generally, the subspace topology induced on the algebraic core coincides with the Scott topology of the algebraic core. The precise statements are collected below.

\begin{Prop}\label{whole-subobject-space}
Let $A\in\mathscr A$. The following conditions are equivalent.
\begin{enumerate}
    \item The categorical Zariski topology on $\sub A$ is $T_0$.
    \item The lattice $\sub A$ is algebraic.
    \item Every subobject $T\subseteq A$ is the supremum of the finitely generated subobjects contained in $T$.
    \item The categorical Zariski topology on $\sub A$ coincides with the Scott topology.
    \item The space $\sub A$, endowed with the categorical Zariski topology, is spectral.
\end{enumerate}
\end{Prop}

\begin{proof}
For $T\in\sub A$, set
$$
T^f:=\bigvee\{F\subseteq T\mid F\text{ is finitely generated}\}.
$$
For every finitely generated subobject $F$ of $A$,
$$
F\subseteq T
\quad\Longleftrightarrow\quad
F\subseteq T^f.
$$
Hence $T$ and $T^f$ belong to exactly the same basic open subsets. If the topology is $T_0$, then $T=T^f$, proving~(3). By Proposition~\ref{compact=fg}, finitely generated subobjects are exactly the compact elements of $\sub A$, so~(3) is equivalent to~(2).

If $\sub A$ is algebraic and $T_1\neq T_2$, then, up to exchanging them, $T_1\nsubseteq T_2$. Since $T_1$ is the supremum of its compact subobjects, there exists a finitely generated subobject $F\subseteq T_1$ such that $F\nsubseteq T_2$. Thus $\Cal B_F$ separates $T_1$ and $T_2$, proving~(1).

For an algebraic lattice, the principal upper sets generated by compact elements form a basis of the Scott topology. Proposition~\ref{compact=fg} therefore gives the equivalence of~(2) and~(4). The Scott space of an algebraic lattice is spectral \cite[Chapter~II]{Gierz2003}, so~(4) implies~(5), while~(5) implies~(1).
\end{proof}

The following example shows that the situation may change if, when defining the subbasis, we replace finitely generated objects with $\lambda$-generated objects (for a regular cardinal $\lambda > \aleph_0$).

\begin{example}\label{powerset-example}
Let $\mathscr A=\mathsf{Set}$ and let $A=\mathbb N$. Then $\sub A$ is canonically isomorphic to $\Cal P(\mathbb N)$, ordered by inclusion. Its compact elements, equivalently its finitely generated subobjects, are precisely the finite subsets. Hence the categorical Zariski topology is generated by the family
$$
\Cal B_F:=\{C\subseteq\mathbb N\mid F\subseteq C\},
$$
where $F$ is finite. Since $\Cal P(\mathbb N)$ is algebraic, this topology is spectral by Proposition~\ref{whole-subobject-space}.

Now let $\lambda>\aleph_0$ be a regular cardinal. Every subset of $\mathbb N$ is $\lambda$-compact. Indeed, if $C\subseteq\bigcup\Delta$ for a $\lambda$-directed family $\Delta$, choose $D_n\in\Delta$ with $n\in D_n$ for every $n\in C$. Since $|C|\leq\aleph_0<\lambda$, there exists $D\in\Delta$ containing all the $D_n$, and hence $C\subseteq D$.

Consequently, the topology generated by the principal upper sets associated with the $\lambda$-compact elements is the topology of all upper subsets of $\Cal P(\mathbb N)$. 

We want to show that $X$ is not sober. Consider the closed subset $\mathcal{F}=\{A\subseteq\mathbb{N}\mid A\text{ is finite}\}\subseteq X$ and notice that $\mathcal{F}$ is irreducible. Indeed, let $U$ and $V$ be open upper subsets of $X$ such that $U\cap\mathcal{F}\neq\varnothing$ and $V\cap\mathcal{F}\neq\varnothing$. Choose $A\in U\cap\mathcal{F}$ and $B\in V\cap\mathcal{F}$. The set $A\cup B$ is finite and $A\cup B\in U\cap V\cap\mathcal{F}$. Thus any two open subsets meeting $\mathcal{F}$ have an intersection that also meets $\mathcal{F}$.

On the other hand, $\mathcal{F}$ is not the closure of any point. For every $A\subseteq\mathbb{N}$, the closure of $\{A\}$ is $\overline{\{A\}}=\{C\subseteq\mathbb{N}\mid C\subseteq A\}$. If $A$ is finite, this closure does not contain all finite subsets of $\mathbb{N}$. If $A$ is infinite, the closure contains the infinite set $A$ itself. Hence $\overline{\{A\}}\neq\mathcal{F}$ for every $A\subseteq\mathbb{N}$.

It follows that $X$ is not sober and therefore is not a spectral space. Thus the topology generated by the compact elements is spectral, whereas the topology generated by the $\lambda$-compact elements is $T_0$ but not spectral.
\end{example}

\subsection{The categorical algebraic core}

Let $B\subseteq A$ be subobjects. Following Banerjee \cite{Banerjee}, set
$$
\Fin(B,A):=
\{T\in\sub A\mid B\subseteq T\text{ and }T=B\vee T^f\},
$$
where
$$
T^f:=\bigvee\{F\subseteq T\mid F\text{ is a finitely generated subobject}\}.
$$

\begin{Prop}\label{categorical-algebraic-core}
Let $A\in\mathscr A$ and let $B\subseteq A$. Then
$$
\Fin(B,A)=\Alg_B(\sub A).
$$
Equivalently,
$$
\Fin(B,A)
=
\left\{
B\vee\bigvee\Cal F
\ \middle|\ 
\Cal F\text{ is a family of finitely generated subobjects of }A
\right\}.
$$
In particular, $\Fin(B,A)$ is an algebraic lattice, its compact elements are precisely the subobjects $B\vee F$, with $F$ finitely generated, the categorical Zariski subspace topology on $\Fin(B,A)$ coincides with its Scott topology and $\Fin(B,A)$ is spectral.
\end{Prop}

\begin{proof}
It follows from Theorem~\ref{abstract-algebraic-core} and Proposition~\ref{compact=fg}.
\end{proof}

\begin{remark}
    This construction extends Banerjee's construction, as well as \cite[Proposition~2.1]{Banerjee}, from the setting of AB5 abelian categories to the present more general categorical setting.
\end{remark}

\begin{corollary}\label{interval-algebraic}
If $\sub A$ is algebraic, then
$$
\Fin(B,A)=[B,A]
$$
for every $B\subseteq A$. In particular, every interval $[B,A]$, endowed with the categorical Zariski topology, is spectral.
\end{corollary}

\begin{proof}
If $T\in[B,A]$, algebraicity and Proposition~\ref{compact=fg} give $T=T^f$. Hence $T=B\vee T^f$ and $T\in\Fin(B,A)$.
\end{proof}

\section{Examples and applications}

\subsection{Classical algebraic examples}

If $\mathscr A$ is locally finitely presentable, the subobjects of every object form an algebraic lattice; see \cite[Theorem~5]{Porst2011}. Proposition~\ref{whole-subobject-space} and Corollary~\ref{interval-algebraic} therefore apply to a large class of familiar examples.

\begin{corollary}\label{variety-example}
Let $\Cal V$ be a finitary variety of algebras, let $A\in\Cal V$, and let $B$ be a subalgebra of $A$. Then the set $[B,A]$ of all intermediate subalgebras, endowed with the categorical Zariski topology, is a spectral space.
\end{corollary}

\begin{proof}
Every finitary variety of algebras satisfies Setting~\ref{setting} for $\lambda=\aleph_0$, and $\sub A$ is an algebraic lattice. Apply Corollary~\ref{interval-algebraic}.
\end{proof}

The same argument applies to submodules, subgroups, subrings, subpresheaves, simplicial subsets, subrepresentations of a quiver, and subcomplexes of a fixed chain complex. In each case, the basic open subsets prescribe the containment of finitely many generators.

\begin{example}\label{sheaves-algebraic}
Let $X$ be a topological space admitting a basis of quasi-compact open subsets, and let $\mathscr A=\Sh(X)$. If $1$ denotes the terminal object, there is a canonical isomorphism of complete lattices
$$
\operatorname{Sub}_{\mathscr A}(1)\cong\Cal O(X).
$$
Under this identification, finitely generated subobjects correspond to quasi-compact open subsets. Since these form a basis of $X$, the frame $\Cal O(X)$ is algebraic. Therefore, for every open subset $B\subseteq X$, the interval
$$
[B,X]:=\{U\in\Cal O(X)\mid B\subseteq U\}
$$
is spectral. A basis of quasi-compact open subsets of this space is given by
$$
\Cal B_K:=\{U\in[B,X]\mid K\subseteq U\},
$$
where $K$ ranges over the quasi-compact open subsets of $X$.

In particular, if $X=\Spec(R)$, one may take $K=D(f_1)\cup\cdots\cup D(f_n)$. Thus the sets
$$
\Cal B_{f_1,\ldots,f_n}
:=
\{U\in[B,\Spec(R)]\mid D(f_1)\cup\cdots\cup D(f_n)\subseteq U\}
$$
form a basis of quasi-compact open subsets.
\end{example}

\subsection{A proper algebraic core and the space of connected components}

Let $X$ be a compact (= quasi-compact and Hausdorff) space. Denote by $\pi_0(X)$ the set of connected components of $X$, endowed with the quotient topology, and let
$$
q:X\longrightarrow\pi_0(X)
$$
be the quotient map. Thus, the points of $\pi_0(X)$ are the connected components of $X$, and a subset $U\subseteq\pi_0(X)$ is open if and only if $q^{-1}(U)$ is open in $X$.

\begin{Prop}\label{components-example}
Let $X$ be a compact space, let $\mathscr A=\Sh(X)$, and let $1$ be the terminal object. Then the quotient map $q\colon X \to \pi_0(X)$ induces an isomorphism of complete lattices
$$
Q\colon\Fin(0,1)\xrightarrow{\sim}\Cal O(\pi_0(X)), \qquad A \mapsto q(A).
$$
It is a homeomorphism when $\Fin(0,1)$ is endowed with the categorical Zariski topology and $\Cal O(\pi_0(X))$ is endowed with the Scott topology.
\end{Prop}

\begin{proof}
We identify $\sub 1$ with $\Cal O(X)$. Since $X$ is a compact space, an open subset is quasi-compact if and only if it is clopen. Hence the compact elements of $\Cal O(X)$ are precisely the clopen subsets of $X$, and
$$
\Fin(0,1)
=
\left\{\bigcup_{i\in I}C_i\ \middle|\ C_i\in\Clop(X)\right\},
$$
where $\Clop(X)$ denotes the family of clopen subsets of $X$. Since $\pi_0(X)$ is a profinite space (see \cite[§5.7, Proposition 5.7.12]{borceux-janelidze}), that is, it is a totally disconnected compact space, it admits a basis of clopen subsets. It follows that the assignment $U \mapsto q^{-1}(U)$ defines a map
$$
Q^*:\Cal O(\pi_0(X))\longrightarrow\Cal \Fin(0,1).
$$
Moreover, $q$ is a quotient map, so $Q^*$ is injective and preserves arbitrary unions and finite intersections; hence it is a homomorphism of lattices. Conversely, every clopen subset of $X$ is a union of connected components, and therefore every $A \in \Fin(0,1)$ is saturated with respect to $q$, that is, $q^{-1}(q(A))=A$. In particular, $q(A) \in \Cal O(\pi_0(X))$ for every $A \in \Fin(0,1)$. It follows that $Q$ and $Q^*$ are mutually inverse isomorphisms of complete lattices. The last assertion is clear.
\end{proof}

\begin{remark}
The construction retains the topology of the component space and forgets the geometry internal to the connected components. In a compact space, the compact elements of $\Cal O(X)$ are precisely the clopen subsets; therefore, the lattice $\Cal O(X)$ is algebraic if and only if the clopen subsets of $X$ form a basis. Thus Proposition~\ref{components-example} may identify a proper algebraic core of $\sub 1$.
\end{remark}

\begin{remark}
A relative version is obtained by fixing a saturated open subset $B=q^{-1}(B_0)$, with $B_0 \in \Cal O (\pi_0(X))$. In this case,
$$
\Fin(B,1)
=
\{q^{-1}(U)\mid B_0\subseteq U\subseteq\pi_0(X),\ U\text{ open}\},
$$
and hence $\Fin(B,1)$ is naturally isomorphic to the interval $[B_0,\pi_0(X)]$ in $\Cal O(\pi_0(X))$.
\end{remark}

\bibliographystyle{plain}
\bibliography{articoli-carmelo-federico}

@book {adamek-rosicky,
    AUTHOR = {Ad\'{a}mek, Ji\v{r}\'{\i} and Rosick\'{y}, Ji\v{r}\'{\i}},
     TITLE = {Locally presentable and accessible categories},
    SERIES = {London Mathematical Society Lecture Note Series},
    VOLUME = {189},
 PUBLISHER = {Cambridge University Press, Cambridge},
      YEAR = {1994},
     PAGES = {xiv+316},
      ISBN = {0-521-42261-2},
   MRCLASS = {18Axx (18-02)},
  MRNUMBER = {1294136},
MRREVIEWER = {J.\ R.\ Isbell},
       DOI = {10.1017/CBO9780511600579},
       URL = {https://doi.org/10.1017/CBO9780511600579},
}

@book{BorceuxHandbook3,
  author    = {Borceux, Francis},
  title     = {Handbook of Categorical Algebra 3: Categories of Sheaves},
  series    = {Encyclopedia of Mathematics and its Applications},
  volume    = {52},
  publisher = {Cambridge University Press},
  address   = {Cambridge},
  year      = {1994},
  isbn      = {9780521441803}
}

@book{borceux-janelidze,
  author    = {Borceux, Francis and Janelidze, George},
  title     = {Galois Theories},
  series    = {Cambridge Studies in Advanced Mathematics},
  volume    = {72},
  publisher = {Cambridge University Press},
  address   = {Cambridge},
  year      = {2001},
  isbn      = {978-0-521-80309-0}
}

@article {Banerjee,
    AUTHOR = {Banerjee A.},
     TITLE = {Closure operators in abelian categories and spectral spaces},
   JOURNAL = {Theory Appl. Categ.},
  FJOURNAL = {Theory and Applications of Categories},
    VOLUME = {32},
      YEAR = {2017},
    NUMBER = {20},
     PAGES = {719–735},
}

@article {finocchiaro-ultrafiltri,
    AUTHOR = {Finocchiaro, Carmelo A.},
     TITLE = {Spectral spaces and ultrafilters},
   JOURNAL = {Comm. Algebra},
  FJOURNAL = {Communications in Algebra},
    VOLUME = {42},
      YEAR = {2014},
    NUMBER = {4},
     PAGES = {1496--1508},
      ISSN = {0092-7872},
       DOI = {10.1080/00927872.2012.741875},
       URL = {http://dx.doi.org/10.1080/00927872.2012.741875},
}

@article {fifolo_transactions,
    AUTHOR = {Finocchiaro, Carmelo A. and Fontana, Marco and Loper, K. Alan},
     TITLE = {The constructible topology on spaces of valuation domains},
   JOURNAL = {Trans. Amer. Math. Soc.},
  FJOURNAL = {Transactions of the American Mathematical Society},
    VOLUME = {365},
      YEAR = {2013},
    NUMBER = {12},
     PAGES = {6199--6216},
      ISSN = {0002-9947},
       DOI = {10.1090/S0002-9947-2013-05741-8},
       URL = {http://dx.doi.org/10.1090/S0002-9947-2013-05741-8},
}

@article{top-vers-null,
	title = {A topological version of {H}ilbert's {N}ullstellensatz},
	journal = {J. {A}lgebra},
	volume = {461},
	pages = {25-41},
	year = {2016},
	issn = {0021-8693},
	doi = {https://doi.org/10.1016/j.jalgebra.2016.04.020},
	url = {https://www.sciencedirect.com/science/article/pii/S0021869316300862},
	author = {Carmelo A. Finocchiaro and Marco Fontana and Dario Spirito},
}

@incollection {fi-fo-sp-dist,
	AUTHOR = {Finocchiaro, Carmelo A. and Fontana, Marco and Spirito, Dario},
	TITLE = {New distinguished classes of spectral spaces: a survey},
	BOOKTITLE = {Multiplicative ideal theory and factorization theory},
	SERIES = {Springer Proc. Math. Stat.},
	VOLUME = {170},
	PAGES = {117--143},
	PUBLISHER = {Springer, [Cham]},
	YEAR = {2016},
	DOI = {10.1007/978-3-319-38855-7\_5},
	URL = {https://doi.org/10.1007/978-3-319-38855-7_5},
}

@book{Gierz2003,
  author    = {Gierz, Gerhard and Hofmann, Karl Heinrich and Keimel, Klaus and Lawson, Jimmie D. and Mislove, Michael and Scott, Dana S.},
  title     = {Continuous Lattices and Domains},
  series    = {Encyclopedia of Mathematics and its Applications},
  volume    = {93},
  publisher = {Cambridge University Press},
  address   = {Cambridge},
  year      = {2003},
  isbn      = {978-0-521-80338-0}
}

@article {hochster_spectral,
    AUTHOR = {Hochster, Melvin},
     TITLE = {Prime ideal structure in commutative rings},
   JOURNAL = {Trans. Amer. Math. Soc.},
  FJOURNAL = {Transactions of the American Mathematical Society},
    VOLUME = {142},
      YEAR = {1969},
     PAGES = {43--60},
      ISSN = {0002-9947},
}

@article {olberding_irredundant,
    AUTHOR = {Olberding, Bruce},
     TITLE = {Irredundant intersections of valuation overrings of
              two-dimensional {N}oetherian domains},
   JOURNAL = {J. Algebra},
  FJOURNAL = {Journal of Algebra},
    VOLUME = {318},
      YEAR = {2007},
    NUMBER = {2},
     PAGES = {834--855},
      ISSN = {0021-8693},
       DOI = {10.1016/j.jalgebra.2007.07.030},
       URL = {http://dx.doi.org/10.1016/j.jalgebra.2007.07.030},
}

@book {Popescu-abelian,
    AUTHOR = {Popescu, N.},
     TITLE = {Abelian categories with applications to rings and modules},
    SERIES = {London Mathematical Society Monographs, No. 3},
 PUBLISHER = {Academic Press, London-New York},
      YEAR = {1973},
     PAGES = {xii+467},
   MRCLASS = {18EXX (16A62)},
  MRNUMBER = {340375},
MRREVIEWER = {B. Stenstr\"{o}m},
}

@article{Porst2011,
  author  = {Porst, Hans-E.},
  title   = {Algebraic lattices and locally finitely presentable categories},
  journal = {Algebra Universalis},
  volume  = {65},
  number  = {3},
  pages   = {285--298},
  year    = {2011},
  doi     = {10.1007/s00012-011-0129-0}
}

@book {Sten,
	AUTHOR = {Stenstr\"om B.},
	TITLE = {Rings of Quotients},
	PUBLISHER = {Grundlehren der Mathematischen Wissenschaften, Springer-Verlag, New York-Heidelberg},
	YEAR = {1975},
	PAGES = {},
	MRCLASS = {},
	MRNUMBER = {},
}
\end{document}